\documentclass[leqno,12pt]{article}

\usepackage{epsfig}

\usepackage{amsmath}
\usepackage{amscd}
\usepackage{amsopn}
\usepackage{amsthm}
\usepackage{amsfonts,amssymb}
\usepackage{amsfonts,bbm}
\usepackage{latexsym}

\makeatletter
\@addtoreset{equation}{section}
\makeatother

\newtheorem{lemma}{LEMMA}[section]
\newtheorem{proposition}[lemma]{PROPOSITION}
\newtheorem{corollary}[lemma]{COROLLARY}
\newtheorem{theorem}[lemma]{THEOREM}
\newtheorem{remark}[lemma]{REMARK}

\newenvironment{knownresult}[1][THEOREM]{\begin{trivlist}
\item[\hskip \labelsep {\bfseries #1.}] \itshape}    {\end{trivlist}}

\newcommand{\real}{\mathbbm{R}}
\newcommand{\nat}{\mathbbm{N}}

\renewcommand{\a}{\alpha}
\renewcommand{\b}{\beta}

\newcommand{\g}{\gamma}

\newcommand{\vp}{\varphi}
\newcommand{\ve}{\varepsilon}

\newcommand{\reald}{{\real^d}}

\newcommand{\on}{\quad\text{ on }}
\newcommand{\und}{\quad\mbox{ and }\quad}
\newcommand{\inv}{^{-1}}
\newcommand{\ov}{\overline}

\newcommand{\dist}{\mbox{\rm dist}}

\newcommand{\itemframe}%
{\setlength{\parskip}{10pt}\begin{enumerate} \setlength{\topsep}{10pt}%
\setlength{\itemsep}{15pt}\setlength{\parsep}{5pt}}

\title{Champagne subdomains with unavoidable bubbles}
\author{WOLFHARD HANSEN and IVAN NETUKA
\thanks{Both authors gratefully acknowledge support
by CRC-701, Bielefeld.}}

\date{}
\begin{document}
\maketitle 

\begin{abstract} 
A champagne subdomain of a connected open  set $U\ne\emptyset$ in $\mathbbm R^d$, $d\ge 2$, is obtained
omitting  pairwise disjoint closed balls $\ov B(x, r_x)$, $x\in X$, the bubbles, 
where $X$ is an infinite, locally finite set in $U$.
The union $A$ of  these balls may be unavoidable, that is, Brownian motion,  starting in $U\setminus A$
and killed when leaving~$U$,  may hit~$A$ almost surely or, equivalently, 
$A$ may have harmonic measure one for $U\setminus A$. 
 
Recent publications by  Gardiner/Ghergu ($d\ge 3$) and by Pres ($d=2$) give rather sharp 
answers to the question how small such a set $A$ may be, when $U$~is the unit ball.

In this paper, using a totally different approach,  optimal  results are obtained, results which hold as well
for  arbitrary connected open sets~$U$.

\bigskip

 {

  }
\end{abstract}

\section{Introduction and main results}

Throughout this paper let $U$ denote a non-empty connected  open  set in $\reald$, $d\ge 2$. 
Let us say that a  relatively closed subset  $A$ of $U$  is  \emph{unavoidable},
if~Brownian motion, starting in $U\setminus A$ and killed when leaving $U$,  hits $A$ almost surely or, equivalently,
if~$\mu_y^{U\setminus A}(A)=1$,  for every $y\in U\setminus A$, where $\mu_y^{U\setminus A}$ 
denotes the harmonic measure at~$y$  with respect to~$U\setminus A$ (we note that $\mu_y^{U\setminus A}$ may fail to be a probability measure,
if~$U\setminus A$ is not bounded).

For $x\in \reald$ and $r>0$, let $B(x,r)$ denote the open ball of center~$x$  and radius~$r$. 
Suppose that $X$ is a countable set in~$U$ having no accumulation point in $U$, and let $r_x>0$, $x\in X$, such that  
the closed balls $\ov B(x,r_x)$,  the \emph{bubbles}, are pairwise disjoint, $\sup_{x\in X} r_x/\dist(x,\partial U)<1$ and,
if $U$ is unbounded, $r_x\to 0$ as $x\to \infty$.   
Then the union $A$ of all $\ov B(x,r_x)$ is relatively closed in $U$, and the connected open set $U\setminus A$ 
(which is non-empty!) is  called a~\emph{champagne subdomain of~$U$}.   

This generalizes the notions used in~\cite{akeroyd, gardiner-ghergu, odonovan, ortega-seip,pres}
for $U=B(0,1)$; see also \cite{carroll} for the case, where $U$ is $\reald$, $d\ge 3$.   
Avoidable unions of  randomly distributed balls have been discussed in \cite{lundh-percolation} and, recently,
in \cite{cardoce}.

It will be convenient  to introduce the set $X_A$ for a champagne subdomain $U\setminus A$: 
$X_A$ is the set of centers of all 
the  bubbles forming $A$ (and $r_x$, $x\in X_A$, is the radius of the bubble centered at $x$).
It is fairly easy to see that, given a champagne subdomain $U\setminus A$ and a finite subset 
$X'$ of $X_A$, the set $A$ is unavoidable if and only if the union of all bubbles $\ov B(x,r_x)$,
$x\in X_A\setminus X'$, is unavoidable.

The main result of Akeroyd \cite{akeroyd} is, for a given $\delta>0$,  the existence of  a champagne subdomain of the unit disc
such that 
\begin{equation}\label{aker}
\mbox{$\sum\nolimits_{x\in X_A}\  r_x<\delta$ and yet  $A$ is unavoidable.}
\end{equation} 
Ortega-Cerd\`a and Seip \cite{ortega-seip} improved
the result of Akeroyd in characterizing a certain class of champagne subdomains~$B(0,1)\setminus A$, 
where $A$ is unavoidable and \hbox{$\sum_{x\in X_A} r_x<\infty$},  and hence the statement of (\ref{aker}) 
can be obtained omitting finitely many of the discs $\ov B(x,r_x)$, $x\in X_A$.

Let us note that already in \cite{HN-sigma} the existence of a champagne subdomain  of an arbitrary 
bounded connected open set   $U$ in $\real^2$ having property (\ref{aker})  was crucial for the construction of an example
answering Littlewood's one circle problem to the negative. 
In fact, Proposition 3 in \cite{HN-sigma} is a bit stronger:
Even a Markov chain formed by jumps on annuli hits $A$ before it goes to $\partial U$. The statement about
harmonic measure (hitting by Brownian motion) is obtained by the first part of the proof of Proposition 3 in \cite{HN-sigma}
(cf.~also \cite{H-kouty}, where this is explicitly stated at the top of page~72). 
This part  uses only  ``one-bubble-estimates''  for the global Green function and the minimum principle.

Recently, Gardiner/Ghergu \cite[Corollary 3]{gardiner-ghergu} proved the following. 

\begin{knownresult}[THEOREM A]  If $d\ge 3$, then, for all $\a>d-2$  and $\delta>0$,
there is a~champagne subdomain $B(0,1)\setminus A$   such that $A$ is unavoidable 
and 
$$ 
\sum\nolimits_{x\in X_A}\  r_x^\a <\delta.
$$ 
\end{knownresult}

Moreover, Pres \cite[Corollary 1.3]{pres} showed  the following for the plane.

\begin{knownresult}[THEOREM B]
If $d=2$, then, for all $\a>1$ and $\delta>0$, there is a~champagne subdomain $B(0,1)\setminus A$ 
 such that $A$ is unavoidable and 
$$
\sum\nolimits_{x\in X_A} \ \bigl(\log\frac 1{r_x}\bigr)^{-\a}   <\delta.
$$ 
\end{knownresult}

Due to capacity reasons both results are sharp in the sense that $\a$ cannot be replaced by $d-2$ in Theorem A      
and $\a$ cannot be replaced by $1$ in Theorem B.  
In~fact, taking $\a=d-2$, $\a=1$, respectively, the corresponding series  diverge, if~$A$~is an unavoidable
set of bubbles (see \cite[p.\ 323]{gardiner-ghergu}  and \cite[Remark 1.4]{pres}). 
The proofs of Theorems A and B are quite involved and, in addition, use the delicate results
\hbox{\cite[Theorem 1]{essen}}  (cf.\ \hbox{\cite[Corollary 7.4.4]{aikawa-essen})} on minimal thinness of subsets~$A$ of~$B(0,1)$ at  points 
$z\in\partial  B(0,1)$
and \hbox{\cite[Proposition 4.1.1]{aikawa-borichev}} on quasi-additivity of capacity. 

Carefully choosing bubbles centered at concentric spheres, estimating related potentials, and using the minimum principle,
we obtain the following optimal result, not only for the unit ball, but even for arbitrary connected open sets.

\begin{theorem} \label{main23}
Let $U\ne \emptyset $ be a connected open  set  in $\reald$, $d\ge 2$, and let 
\hbox{$h:(0,1)\to (0,1)$}  be such that $\lim_{t\to 0} h(t)=0$. 
  Then, for every $\delta>0$,
there is a~champagne subdomain $U\setminus A$        such that $A$ is unavoidable and     
\begin{eqnarray*} 
\sum\nolimits_{x\in X_A} \ \bigl(\log \frac 1{r_x}\bigr)\inv  h(r_x)  &<&\delta, \qquad  \mbox{ if } d=2,\\[1mm]
\sum\nolimits_{x\in X_A} \ r_x^{d-2} h(r_x) &<&\delta, \qquad \mbox{ if } d\ge 3. 
\end{eqnarray*} 
   \end{theorem} 

Moreover, we may treat the cases $d=2$ and $d\ge 3$ simultaneously.
To that end we define functions 
\begin{equation*} 
 N(t) :=\begin{cases} \log \frac 1t  ,&\quad\mbox{ if } d=2,\\
                                t^{2-d} ,&\quad\mbox{ if }d\ge 3,
                   \end{cases}                   \und \vp(t):= 1/N(t)
\end{equation*} 
so that  $(x,y) \mapsto N(|x-y|)$   is the global Green function  and, for $d\ge 3$, $\vp(t)=t^{d-2}$ is the capacity of  balls
with radius $t$ (for $d=2$, $\vp(t)$  should  be considered for $t\in (0,1)$ only). Using the (capacity) function $\vp$
our Theorem \ref{main23} adopts the following form.

\begin{theorem} \label{main}
Let $U\ne \emptyset $ be a connected open  set  in $\reald$, $d\ge 2$, and let \hbox{$h:(0,1)\to (0,1)$} be such that $\lim_{t\to 0} h(t)=0$. 
 Then, for every $\delta>0$,
there is a~champagne subdomain $U\setminus A$  
such that $A$  is unavoidable and 
\begin{equation}\label{sharp}
\sum\nolimits_{x\in X_A}  \vp(r_x) h(r_x)  < \delta.
\end{equation} 
 \end{theorem} 
       
Accordingly,  the results by Gardiner/Ghergu and Pres (Theorems A and B) can be unified as follows.    

\begin{knownresult}[THEOREM C] If $d\ge 2$, then, for all $\ve >0$ and $\delta>0$, there is a~champagne subdomain $B(0,1)\setminus A$ 
such that $A$ is unavoidable and 
$$
\sum\nolimits_{x\in X_A} \ \vp(r_x)^{1+\ve}     < \delta.
$$
\end{knownresult}

Clearly, Theorem C follows from Theorem \ref{main} taking $h=\vp^\ve$. Of course, we may get much stronger statements
taking, for example,  
$$
h(t) =(\log \log \dots \log (1/\vp(t)))\inv, \qquad \mbox{ $t>0$ sufficiently small.}
$$

In fact, we shall obtain the following result for the open unit ball.

\begin{theorem}\label{dream-unit}
Let $d\ge 2$, $\delta>0$,  and \hbox{$h:(0,1)\to (0,1)$}  with \hbox{$\lim_{t\to 0} h(t)=0$}. Further, 
let $(R_k)$ be a sequence in $(1/2,1)$ which is strictly increasing to $1$. 

Then there exist finite sets $X_k$ in $\partial B(0,R_k)$
and $0<r_k<(1-R_k)/6$ such that, taking 
$$
 A:=\bigcup\nolimits_{x\in X_k, k\in\nat}  \ov B(x,r_k),
$$
the set $B(0,1)\setminus A$ is a~champagne subdomain,
$A$ is unavoidable and  {\rm (\ref{sharp})} holds.
\end{theorem} 

Let us finish this section explaining  in some detail how these results are obtained.
Given an exhaustion of an arbitrary domain $U$ by a sequence $(V_n)$ of bounded open subsets,
we first present a criterion for unavoidable sets $A$ in $U$  in terms of probabilities for Brownian motion, starting in $\ov V_n$, 
to hit $A$ before leaving $V_{n+1}$  (Section~2). 

To apply this criterion we prove the existence of  $c>0$ and $\kappa>0$ such that the following holds (Sections 3 and 5):
 Given  $1/2<R<1$  and $0< \rho\le 1/3$, there exists 
$0<\rho_0\le \rho/3$  such that,  for \emph{every} $0<r< \rho_0$,
we may choose a~finite subset~$X_r$ of~$\partial B(0,R)$ satisfying
\begin{itemize}
\item [\rm (i)] the product $\# X_r\cdot \vp(r)$ is bounded  by~$c_{}\rho\inv$,
\item [\rm (ii)] the balls $\ov B(x,r)$, $x\in X_r$, are pairwise disjoint,
\item [\rm (iii)] starting in \hbox{$\ov B(0,R+\rho)$} Brownian motion hits the union of the balls $\ov B(x,r)$, $x\in X_r$, before leaving 
$B(0,R+2\rho)$ with a probability which is at least~$\kappa$.
\end{itemize} 
In Section 4 we give a straightforward application of our construction $X_r$ to the unit ball considering an exhaustion
$(B(0,R_k))_{k\ge k_0}$ given by $R_{k+1}-R_k= (k\log^2 k)\inv$ and a~``one-bubble-estimate'' for the global Green function.
The resulting   Proposition~\ref{bubble} is already fairly close to Theorem~C.

The proof of (iii)  in Section \ref{section-crucial} will be based on a comparison of 
 the sum of the potentials for the points $x\in X_r$  with the equilibrium potential for $\ov B(0,R)$ 
 (both with respect to $B(0,R+2\rho)$). The proof of Theorem~\ref{dream-unit} is now easily accomplished (Section~6).
Indeed, given $R_k\uparrow 1$, it suffices to take $\rho_k:=(R_{k+1}-R_k)/2$ and to choose
$0<r_k< \rho_{0,k} \le \rho_k$ with $c\rho_k\inv h(r_k)<2^{-k}\delta$ and $r_k\le (R_k-R_{k-1})/2$.

Finally, using the ingredients of this proof, we obtain Theorem~\ref{main}  in full generality (Section~7).

\section{A general criterion for unavoidable sets} \label{one-bubble}

 Given an open set $W$ in $\reald$ and a  bounded Borel measurable function~$f$ on $\reald$,   
let $H_Wf$ denote the function which extends  the (generalized) Dirichlet solution  \hbox{$x\mapsto \int f\,d\mu_x^W$}, 
$x\in W$, to a function on $\reald$ taking the values $f(x)$ for $x\in\reald\setminus W$.
We shall use that the harmonic kernel $H_W$ has the following property: If $W'$ is an open set in $W$,
then $H_{W'}H_W=H_W$. 

Let $U\ne \emptyset$ be a connected open set  in $\reald$, $d\ge 2$, and let $A\subset U$ be  relatively closed. 
Then $A$ is unavoidable if and only if
\begin{equation*} 
H_{U\setminus A}1_A=1 \on\  U.
\end{equation*}

\begin{proposition}\label{nlogn}
Let $0\le \kappa_j\le 1$ and  $V_j$ be bounded open sets in $U$,   $j\ge j_0$,              
such that $\ov V_j\subset V_{j+1}$, $V_j\uparrow U$,  and the following holds:  For every   $j\ge j_0$     
and every $z\in \partial V_j\setminus A$, there exists a closed set~$E$ in~$A\cap V_{j+1}$ such that 
\begin{equation}\label{HUE}
H_{V_{j+1}\setminus E} 1_{E}(z)\ge     \kappa_j.
\end{equation} 
Then, for all $n,m\in \nat$, $j_0\le n<m$, 
\begin{equation}\label{criterion}
  H_{U\setminus A} 1_A\ge 1-\prod\nolimits_{n\le j<m} (1-\kappa_j) \on\ \ov V_n. 
\end{equation} 
In particular,  $A$ is unavoidable if the series $\sum\nolimits_{j\ge j_0}\kappa_j$ is divergent.
\end{proposition}

As we noticed later on, the probabilistic aspect of such a result has already been used 
in \cite{ortega-seip} and subsequently in \cite{carroll,odonovan}: Of course, Brownian motion starting in $V_n$ hits
$\partial V_n$ before reaching~$\partial V_{n+1}$.  Inequality (\ref{HUE}) implies  that a Brownian
particle starting at some $z\in\partial V_j\setminus A$, $n\le j<m$,  does  not hit $A$ 
before reaching~$\partial V_{j+1}$  with probability at most  $1-\kappa_j$. 
By induction and by the strong Markov property, it does not hit $A$ with probability at most 
$\prod\nolimits_{n\le j<m} (1-\kappa_j)$  before reaching $\partial V_m$, and therefore  
it  hits $A$ with probability at least $1-\prod\nolimits_{n\le j<m} (1-\kappa_j)$ before leaving $U$. 

\begin{proof}[Proof of Proposition  \ref{nlogn}] For $j\ge j_0$, let  $W_{j+1}:=V_{j+1}\setminus A$. If $E$ is a  closed set  in~$A\cap V_{j+1}$, 
then $H_{W_{j+1} }1_{\partial V_{j+1}}\le 1- H_{V_{j+1}\setminus E} 1_{E}$, by  the minimum principle. Hence, by~(\ref{HUE}), 
\begin{equation*} 
H_{W_{j+1} }1_{\partial V_{j+1}} \le  1-\kappa_j \on \partial V_j.
\end{equation*} 
Now let $n,m\in\nat$, $j_0\le n<m$. By induction,
\begin{equation*} 
    H_{W_m} 1_{\partial V_m}
    =H_{W_{n+1} }H_{W_{n+2}}\dots H_{W_m} 1_{\partial V_m}
           \le \prod\nolimits_{n\le j<m} (1- \kappa_j) \on \partial  V_n.
 \end{equation*} 
By the minimum principle,  we conclude that 
\begin{equation*} 
H_{U\setminus A} 1_A\ge      H_{W_m} 1_{ A} \ge  1-   H_{W_m} 1_{\partial V_m}\ge  1- \prod\nolimits_{n\le j<m} (1- \kappa_j) \on\  \ov V_n.
\end{equation*} 
\end{proof}

\section{Choice of bubbles and crucial estimate} \label{choice}

Let $R\in (1/2,1)$, $U=B(0,R)$,  and $\rho \in (0,1/3)$. 
% (1-R)/2\le 1/4
For every $r>0$ which is sufficiently small, we shall choose an associated finite subset $X_r$ of $\partial U$ 
and consider the union $E_r$ of all bubbles $\ov B(x,r)$, $x\in X_r$. 
For $r>0$, we first define
\begin{equation}\label{ab-def}
 \b:=(\vp(r)\rho)^{1/(d-1)}.
\end{equation} 
In other words, we take $\b$ satisfying 
% comment: clear, if $d>2$; for d=2 use (\log (1/r))\inv \le (\log (1/\rho))\inv \le (\log 3)\inv <1
\begin{equation}\label{ar}
 \vp(r)=\b^{d-1}\rho\inv , \quad\mbox{ that is, }\quad 
      r=\begin{cases} \exp(-\rho/\b),&\quad\mbox{ if } d=2,\\[1.5mm]
                                        \b^{(d-1)/(d-2)}\rho^{-1/(d-2)},&\quad\mbox{ if }d\ge 3.
                   \end{cases}
\end{equation} 
It is easily seen that $\b<\rho$, if $r<\rho$. There exists $\rho_0\le \rho/3$ such that
\begin{equation}\label{rnan}
  r<  \b/3\, ,\qquad\mbox{ whenever } r\in(0, \rho_0) .
\end{equation} 
Indeed, if $d\ge 3$ and $r<3^{1-d}\rho$, then $r/\b=\bigl(r^{d-1}/(r^{d-2}\rho)\bigr)^{1/(d-1)}< 1/3$.
Assume  now that $d=2$ and $r<(1/18)\rho^2$. Then $\rho/\b=\log (1/r)<\log\{[\rho/(3r)]^2/2\}<\rho/3r$.

Given $0<r<\rho_0$, we choose  a finite subset $X_r$ of $\partial U$   such that the balls $B(x,\b)$, $x\in X_r$,        
 cover~$\partial U$ and the balls $B(x,\b/3)$, $x\in X_r$, are pairwise disjoint (such a~set~$X_r$ exists; see \hbox{\cite[Lemma 7.3]{Rudin}}).
%Let $E_r$ denote the union of all bubbles $\ov B(x,r)$, $x\in X_r$. 
By (\ref{rnan}), the balls $\ov B(x,r)$, $x\in X_r$, forming~$E_r$ are pairwise disjoint.
A~consideration of the  areas involved,  when intersecting the balls with~$\partial U$, 
shows that there exists a constant $c_{}=c_{}(d)>0$ such that 
\begin{equation}\label{area}
c_{}\inv \b^{1-d}\le  \# X_r \le c_{}\b^{1-d}
\end{equation} 
and hence, by (\ref{ar}), 
\begin{equation}\label{log-estimate}
% c_{}\inv b\le 
c\inv \rho\inv \le \# X_r \cdot \vp(r) \le c_{} \rho\inv,
\end{equation} 
that is, the sum of the capacities of the bubbles $\ov B(x,r)$, $x\in X_r$, is approximately~$\rho\inv$.
Let us stress already now that, by  (\ref{log-estimate}), for any choice of $\rho \in (0,1/3) $, the product 
$$
\#X_r\cdot \vp(r)\, h(r)
$$
 is arbitrarily small  provided  $r$ is small enough. Defining
$$
     U':=B(0,R+\rho) \und V:=B(0,R+2\rho)
$$
the following Proposition~\ref{heart} (proved in Section \ref{section-crucial}) will hence quickly lead to Theorem~\ref{dream-unit}.

\begin{proposition}\label{heart}
There exists a constant $\kappa=\kappa(d)>0$ such that
\begin{equation}\label{HV-universal}
 H_{V\setminus E_r}1_{E_r}\ge \kappa   \on \ov{U'}, \qquad \mbox{ for \emph{every} }   r\in (0,\rho_0),
\end{equation} 
that is, Brownian motion starting in $\ov {U'}$ hits $E_r$ with
probability at least $\kappa$  before leaving $V$, whatever $0<r<\rho_0
$ is.
\end{proposition}

\section{Result based on  a ``one-bubble-approach''}   

It may be surprising that, having Proposition \ref{nlogn} and our construction of unions~$E_r$ of bubbles  centered 
at spheres $B(0,R)$, 
already a ``one-bubble-approach'',  which only uses the global Green function with one pole, 
immediately yields a  result  which   is  almost as strong as Theorem~C.

For Proposition \ref{bubble}, a sequence $(R_k)_{k\ge k_0}$ will be chosen in the following way.
We fix $k_0\ge 3^{d-1}$ such that $\sum_{j\ge k_0} (j \log^2 j)\inv <1/2$ and $e^{-k}<(9k\log^2k)\inv$, 
for $k\ge k_0$. For every $k\ge k_0$,  let 
$$
    R_k:=1-\sum\nolimits_{j\ge k} (j \log^2j)\inv, \quad  U_k:=B(0,R_k), \quad V_k:=B(0, R_k+ (2 k\log^2k)\inv).
$$
 To apply our construction in Section 3 let us, for the moment, fix $k\ge k_0$
and let $R:=R_k$, $\rho:=(3\log^2k)\inv< 1/3$ so that $U=U_k$, $U'=V_k$, and $V=U_{k+1}$.
Further, let 
\begin{equation*} 
  r:=\begin{cases} e^{-k},&\quad\mbox{ if } d=2,\\[1.5mm]
                             k^{-(d-1)/(d-2)}\rho,&\quad\mbox{ if }d\ge 3.
                   \end{cases}
\end{equation*} 
Then $\b:=\rho/k$ satisfies $\vp(r)= \b^{d-1}\rho\inv$ and $r<\b/3<\rho/3$.  
So we may choose a~corresponding finite set $X_r$ and take $X_k:=X_r$, $r_k:=r$.    
Let us already notice that $r_k/(1-R_k)< \rho/(3k) \cdot k\log^2k\le 1/9$.

\begin{proposition}\label{bubble}  Let $\ve>1/(d-1)$ and $\delta>0$. Then 
there exists $K\ge k_0$ such that,  taking 
$$
 A:=\bigcup\nolimits_{x\in X_k, k\ge K}  \ov B(x,r_k),
$$
the set $B(0,1)\setminus A$ is a~champagne subdomain,
$A$ is unavoidable, and  
\begin{equation}\label{s-weak}
          \sum\nolimits_{x\in X_A}  \vp(r_x)^{1+\ve}\ < \  \delta. 
\end{equation}                               
\end{proposition}

\begin{proof}  Let $k\ge k_0$. By (\ref{ar}), $\vp(r_k)\le k^{1-d} $. Hence, by (\ref{log-estimate}), 
$$
\# X_k \, \vp(r_k)^{1+\ve} \le c_{}(3\log^2 k)  \vp(r_k)^\ve \le   c_{}(3\log^2 k)   k^{\ve(1-d)}.
$$
So (\ref{s-weak}) holds,   if $K$ is sufficiently large.              

We next claim that the union $A$ of all $\ov B(x,r_k)$, $x\in X_k$, $k\ge K$, is unavoidable. Indeed, let  us fix $k\ge K$ and let $\b, r$ be  as above. 
 Let $z\in \partial U_k\setminus A$. 
There exists $x\in  X_k$ such that $|z-x|<\b$. We define $E:=\ov B(x,r)$ and 
\begin{equation*} 
                           g(y):=\vp(r) \bigl(N(|y-x|)-N(3\b) \bigr) , \qquad y\in\reald.
\end{equation*} 
Since $3\b<(k\log^2k)\inv$, we know that $B(x,3\b)\subset U_{k+1}$,  and hence $g\le 0$ on~$ \partial U_{k+1}$. 
Further,  $g\le \vp(r)N(r)=1$ on the boundary of~$E$.   By the minimum principle,
$$
           H_{U_{k+1}\setminus E} 1_E\ge g \on U_{k+1}\setminus E.
$$
Clearly,  $N(|z-x|)-N(3\b)\ge (2/3)\b^{2-d}$ , since $\log 3\ge 1$ and $1-3^{2-d}\ge 2/3$ for $d\ge 3$.
Therefore, by (\ref{ar}), 
$$
H_{U_{k+1}\setminus E} 1_E(z)\ge g(z)\ge (2/3)\vp(r)\b^{2-d}=(2/3)\b/\rho=(2/3) k\inv.
$$
By Proposition \ref{nlogn}, $A$~is unavoidable. Clearly, $B(0,1)\setminus A$ is a champagne subdomain.
\end{proof} 

\begin{remark}{
If $d\ge 3$, then $\vp(r_x)^{1+\ve}=r_x^{(d-2)(1+\ve)}$, where the critical exponent
$(d-2)(1+1/(d-1))=d-1-1/(d-1)$ is strictly smaller than $d-1$.
}
\end{remark}

\section{Proof of the crucial estimate} \label{section-crucial}

For a proof of Proposition \ref{heart} let us now return to the general setting of  Section \ref{choice},
where $0<R<1/2$, $0<\rho< 1/3$, $U=B(0,R)$, $U'=B(0,R+\rho)$, $V=B(0,R+2\rho)$,
and  let $G$ be the Green function for $V$.
Further, we have $0<\rho_0\le \rho/3$ such that  $3r<\b=(\vp(r) \rho)^{1/(d-1)}$, for every $0<r<\rho_0$.

\begin{lemma}\label{zzy} There exists a constant $c_1:=c_1(d)>0$ such that 
\begin{equation*}
G(y,z)\le c_1 G(y,z'), \quad \mbox{ if }  y\in V \mbox{ and } z,z'\in \partial U  \mbox{ with } |y-z'|\le 4 |y-z|.
\end{equation*} 
\end{lemma} 

\begin{proof} For $y,z\in V$, let   $\Psi(y,z) :=(R+2\rho-|y|)(R+2\rho-|z|)/|y-z|^2$ and  
\begin{equation*}
F(y,z):= \begin{cases}   \log \bigl(1+\Psi(y,z)\bigr),&\quad d=2,\\
                                               \min\{1, \Psi (y,z)\} |y-z|^{2-d},&\quad d\ge 3.
                         \end{cases}
\end{equation*} 
If $y\in V$ and  $z,z'\in \partial U$ with $|y-z'|\le 4|y-z|$, then $\Psi(y,z)\le 4^2 \Psi(y,z')$,
and hence $F(y,z)\le 4^d F(y,z')$.
It follows immediately from \cite[Theorem 4.1.5]{gardiner-armitage} that there exists a constant $c_0=c_0(d)$ 
such that $c_0\inv F\le G\le c_0 F$. So it suffices to take $c_1:=4^d c_0^2$. 
\end{proof}

For every measure $\chi$ on $V$, let
$G\chi(y):= \int G(y,z)\,d\chi(z)$,  $y\in V$. 
Let  $\sigma$ be the normalized surface measure on $\partial U$. We note that 
\begin{equation}\label{gs-formula}
G(\cdot,0)=N(|\cdot|)-N(R+ 2\rho), \quad  G\sigma=\min\{G(\cdot,0), N(R) -N(R+2\rho)\}.
\end{equation} 

Now we fix $r\in (0,\rho_0)$ and define 
$$
        \mu:=\b^{d-1}\sum\nolimits_{x\in X_r} \ve_x.
$$
Since $c_{}\inv \le \|\mu\|\le c_{}$, by (\ref{area}),  and $X_r$ is  distributed on $\partial U$ in a fairly regular way, there is a close relation
between $G\mu$ and $G\sigma$. We shall use the following.

\begin{lemma}\label{Gsigmamu}
There exists a constant $C=C(d)>0$ such that $G\sigma\le C G\mu$ on $\partial U'$ and,  for every $x\in X_r$,
\begin{equation*} 
G\mu\le  \b^{d-1}G(\cdot, x)+CG\sigma \on \ov B(x,r).
\end{equation*} 
\end{lemma} 

\begin{proof} 
Let us introduce a  partition of $\partial U$ corresponding to $X_r=\{x_1,\dots, x_M\}$. 
For $1\le j\le M$, let $S_j':=\partial U\cap B(x_j,\b/3)$,  $S_j'':=\partial U\cap B(x_j,\b)$, 
and let $S'$ be the union of the pairwise disjoint sets $S_1',\dots,S_M'$.
We recursively define $S_1,S_2, \dots, S_M$ by
$            S_1:=S_1'\cup (S_1''\setminus S')$ and 
$$
            S_j:=\bigl(S_j'\cup (S_j''\setminus S')\bigr)\setminus (S_1\cup \dots \cup S_{j-1}).
$$ 
 Since  $S_1'',\dots,S_M''$ cover $\partial U$, the sets $S_1,\dots, S_M$ form  a partition of $\partial U$  such that 
$$
 S_j'\subset S_j\subset S_j'' \qquad \mbox{ for every }1\le j\le M.  
$$
So there exists a constant $c_2=c_2(d)>0$ such that  
\begin{equation}\label{nm1}
c_2\inv \b^{d-1}\le \sigma(S_j) \le c_2\b^{d-1}, \qquad 1\le j\le M. 
\end{equation}

To prove  the first inequality, we fix $y\in \partial U'$. 
Let $1\le j\le M$. For every $z\in S_j$, 
$|y-z|\ge \rho> \b >|z-x_j|$,  and hence 
 $|y-x_j|\le |y-z|+|z-x_j|< 2|y-z|$. So, by Lemma~\ref{zzy},   $G(y, \cdot)\le  c_1 G(y,x_j)$ on~$S_j$, and  hence 
$$
G(1_{S_j}\sigma)(y)=\int_{S_j} G(y,z)\,d\sigma(z) \le c_1 \sigma(S_j) G(y,x_j) \le c_1c_2\b^{d-1} G(y,x_j).
$$
Taking the sum we see that $G\sigma(y)\le c_1c_2 G\mu(y)$. 

To prove the second inequality let $x:=x_{j_0}$, $1\le j_0\le M$, and assume that  $1\le j\le M$, $j\ne j_0$. 
Moreover,  let $y\in \ov B(x,r)$ and $z'\in S_j$.
Clearly, $y\in  B(x,\b/3)$, by (\ref{rnan}). Since $B(x,\b/3)\cap
B(x_j,\b/3)=\emptyset$, we see that $|y-x_j|>\b/3$, whereas $|x_j-z'|<\b$. So $|y-z'|\le
|y-x_j|+|x_j-z'|< 4 |y-x_j|$, and therefore $G(y,x_j)\le c_1 G(y,\cdot)$ on $S_j$, by
Lemma~\ref{zzy}. By integration, 
$\sigma(S_j)G(y,x_j) \le c_1G(1_{S_j}\sigma) (y)$.  Thus, using (\ref{nm1}), 
\begin{eqnarray*} 
 G\mu& \le&  \b^{d-1}G(\cdot,x) + c_2 \sum\nolimits_{j\ne j_0} \sigma(S_j)  G(\cdot,x_j)\\
         &\le& \b^{d-1}G(\cdot,x) + c_1 c_2\sum\nolimits_{j\ne j_0} G(1_{S_j}\sigma) \\
         &\le& \b^{d-1}G(\cdot,x)  + c_1 c_2 G\sigma \quad  \on  \ov B(x,r).
\end{eqnarray*} 
Taking $C:=c_1c_2$ the proof is finished.
\end{proof}

By (\ref{gs-formula}), there exists a constant $c_3=c_3(d)>0$ such that 
\begin{equation}\label{Gsigma}
c_3\inv \rho \le G\sigma(y)\le  {c_3}\rho\qquad\mbox{ whenever } y\in V \mbox{ such that } |\, |y|-R\,|\le \rho.
\end{equation} 
After these preparations we are  ready to prove the crucial estimate in  Proposition~\ref{heart}. 
We first claim that
\begin{equation}\label{essential}
              G\mu\le (2+c_3 C)\rho    \on \partial E_r.
\end{equation} 
Indeed, let $x\in X_r$ and $y\in \partial B(x,r)$.  
Since $B(0,1)\subset B(x,2)$ and  $|y-x|=r<1/2$, we obtain that $G(y,x)\le N(|y-x|)-N(2)=N(r)-N(2)\le 2 N(r)$
(if~$d=2$, then $N(r)-N(2)=\log (1/r)+\log 2\le 2 \log (1/r)$).
So, by (\ref{ar}),
\begin{equation*} 
                        \b^{d-1}G(x,y)\le 2\b^{d-1}N(r) =2\b^{d-1}\vp(r)\inv=2  \rho. 
\end{equation*} 
Further, by (\ref{Gsigma}),  $G\sigma(y)\le c_3 \rho$.  Therefore
(\ref{essential}) holds, by Lemma \ref{Gsigmamu}.

Since $G\mu$ is harmonic on $V\setminus {E_r}$ and $G\mu$ vanishes at $\partial V$,
we conclude that
\begin{equation*} 
         H_{V\setminus {E_r} }1_{E_r}\ge   (2+c_3C)\inv \rho\inv G\mu  \on V\setminus {E_r}.
\end{equation*} 
On the other hand,  by Lemma \ref{Gsigmamu} and (\ref{Gsigma}),
\begin{equation*} 
 G\mu\ge C\inv  G\sigma  \ge  (c_3C) \inv \rho 
 \on \partial U', 
\end{equation*} 
whence on $\ov {U'}$, by the minimum principle. Taking $\kappa:=(c_3C(2+c_3C))\inv $ we thus obtain that
$  H_{V\setminus {E_r} }1_{E_r}\ge \kappa$ on $\ov{U'} $.

\section{Main result for the open unit ball}

To prove Theorem \ref{dream-unit},  let $\delta>0$ and $h\colon (0,1)\to (0,1)$ be such that $\lim_{t\to 0}h(t)=0$.
It will be convenient to introduce the smallest increasing majorant $\tilde h$ of $h$, given by
$$
                      \tilde h(t):=\sup\{h(s)\colon 0<s\le t\}
$$ 
(of course, we also have $\lim_{t\to 0} \tilde h(t)=0$).  

Further, let $(R_k)$ be a sequence in $(1/2,1)$ which is strictly increasing 
to $1$, and let $R_0:=1/2$. For every $k\in \nat$, let
$$
         U_k:=B(0,R_k) \und V_k:=B(0, (R_k+R_{k+1})/2).  
$$

To apply our construction in Section 3 let us, for the moment, fix $k\in\nat$ and let $R:=R_k$, 
$\rho:= (R_{k+1}-R_k)/2$, and $\rho_0\le \rho/3$ such that (\ref{rnan}) holds.  We observe that $U=U_k$, $U'=V_k$, and $V=U_{k+1}$.
There exists $\eta>0$ such that 
\begin{equation}\label{h-eta}
c\rho\inv\tilde h(\eta)<  2^{-k}\delta.
\end{equation} 
We fix $0<r<\min\{\eta, \rho_0,(R_k-R_{k-1})/2\}$
% . We then choose a corresponding finite set $X_r$ as described in Section \ref{choice} 
and take
$$
               X_k:=X_r \und      r_k:=r.
$$
By (\ref{log-estimate}) and (\ref{h-eta}), 
\begin{equation}\label{sum-k}
          \# X_k\cdot \vp(r_k) \tilde h(r_k) \le c_{} \rho\inv \tilde h(\eta)<2^{-k}\delta.
\end{equation} 
By Proposition \ref{heart},  the union $E_k$ of all $\ov B(x,r_k)$, $x\in X_k$, satisfies
$      H_{U_{k+1}\setminus E_k} 1_{E_k}\ge \kappa $ on~$\ov V_k$, and hence, by the minimum principle,
\begin{equation}\label{HV-k}
      H_{V_{k+1}\setminus E_k} 1_{E_k}\ge \kappa \on \ov V_k.
\end{equation} 

By our choice of $r_k$,    the balls $\ov B(x,r_k)$, $x\in X_k$, $k\in\nat$, 
are pairwise disjoint and $r_k< (R_{k+1}-R_k)/6< (1-|x|)/6$, for every $x\in X_k$.  
Let $A$ be the union  of all~$E_k$, $k\in\nat$. Then  $B(0,1)\setminus A$ is a champagne subdomain. 
By (\ref{HV-k}) and Proposition~\ref{nlogn},  $A$~is unavoidable.
Finally,
$$
  \sum\nolimits_{x\in X_k, k\in\nat}  \vp(r_x) h(r_x) <\delta,
$$
 by  the definition of $\tilde h$ and (\ref{sum-k}), finishing the proof of Theorem \ref{dream-unit}.

For an application in Section \ref{final} let us note the following.

\begin{corollary}\label{y-ball}
Let $y\in \reald$, $0<a'< a\le 1$, $\g\in (0,1)$,  and \hbox{$\delta_y>0$}. Then  there exist a finite set $X_y$
 in  $B(y,a)\setminus \ov B(y,a')$ and  $0<s_x< (a-|x-y|)/6$, $x\in X_y$, such that the  balls $\ov B(x,s_x)$, $x\in X_y$,
are pairwise disjoint,
\begin{equation}\label{y-ball-fill} 
                 \sum\nolimits_{x\in X_y} \vp(s_x) h(s_x) <\delta_y
\und
                  H_{B(y,a)\setminus A_y} 1_{A_y} \ge \g \on \ov B(y,a'),
\end{equation} 
where $A_y$ is the union of all $\ov B(x,s_x)$, $x\in X_y$.
\end{corollary} 

\begin{proof} By translation invariance, we may assume that $y=0$. Let $\tau:=a'/a$. 
 By~(\ref{sum-k}), there exist $m,n\in \nat$, $m>n$,   such that $R_n> \tau$,   
$$
               \sum\nolimits_{n\le k<m } \# X_k\,\vp(r_k) \tilde h(r_k) <\delta_0, \und
                      (1-\kappa)^{(m-n)} \le 1-\g. 
$$

Let $X_0$ be the set of all $ax$,  $x\in X_k$, $n\le k<m$, and, for $x\in X_0$, let $s_x:=a r_{x/a}$. 
Since $\vp$ and $\tilde h$ are increasing, we obtain that 
$$
      \sum\nolimits_{x\in X_0}  \vp(s_x)  h(s_x) \le  \sum\nolimits_{n\le k<m}  \# X_k\,\vp(r_k) \tilde h(r_k) <\delta_0.
$$
By scaling invariance of harmonic measures, by (\ref{HV-k}) and Proposition~\ref{nlogn},   
the second inequality in (\ref{y-ball-fill}) follows
(note that  (\ref{HUE}) trivially holds for $j\ge m$ with $\kappa_j:=0$ and $E:=\emptyset$).    
\end{proof}

\section{Proof for arbitrary connected open sets}\label{final}

Let $U$ be an arbitrary non-empty connected open set in $\reald$, $d\ge 2$.
Let us fix  bounded open sets $V_n\ne\emptyset$, $n\in\nat$, such that $\ov V_n\subset V_{n+1} $
and $V_n\uparrow U$. For every $n\in\nat$, we define 
$$
              d_n:=\min\left\{\dist(\partial V_n, \partial V_{n-1}\cup \partial V_{n+1}), \, 1/n\right\}
$$
(take $V_0:=\emptyset$) and choose a finite subset $Y_n$ of $\partial V_n$ such that the balls $B(y,d_n/2)$, 
$y\in Y_n$, cover $\partial V_n$ and the balls $B(y,d_n/6)$, $y\in Y_n$, are pairwise disjoint. 
For $y\in Y_n$, we apply Corollary \ref{y-ball} with 
\begin{equation}\label{deltan}
a':=\frac {d_n}7, \qquad a:=\frac{d_n}6, \qquad \g:=\frac 12, \qquad \delta_y:=\frac{\delta}{\# Y_n\cdot 2^n}\, .
\end{equation} 

Let $X$ be the union of all $X_y$, $y\in Y_n$, $n\in\nat$, and let $A$ be the union of all $\ov B(x,s_x)$, $x\in X$.
For all $x\in X$,
$s_x<a/6=d_n/18<\dist(x,U^c)/18$ and, if $U$ is unbounded,  $s_x\to 0$ if $x\to\infty$. Of course, $X$ is locally finite in $U$.
Hence, $U\setminus A$ is a~champagne subdomain.  Moreover,  by (\ref{y-ball-fill}) and (\ref{deltan}),
$$
    \sum\nolimits_{x\in X} \vp(s_x) h(s_x ) <\sum\nolimits_{n\in\nat} \sum\nolimits_{y\in Y_n}\delta_y=\delta.
$$

So it remains only to prove $A$ is unavoidable.   To that end we define 
\begin{equation*} 
      \eta:= \inf \bigl\{H_{B(0,1)\setminus \ov B(0,1/7)} 1_{\ov B(0,1/7)} (z)\colon |z|< 1/2\bigr\}
\end{equation*} 
so that Brownian motion starting in $B(0,1/2)$  hits $\ov B(0,1/7)$ with probability at least $\eta$ before leaving $B(0,1)$.
(Of course $\eta$ is easily determined: It is $\log 2/\log 7$, if $d=2$, and
$(2^{d-2}-1)/(7^{d-2}-1 )$, if $d\ge 3$.)
Let us fix $n\in\nat$, $y\in Y_n$, and let $E$ be the union of all $\ov B(x,s_x)$, $x\in X_y$.  
We claim that 
\begin{equation}\label{eta-est}
             H_{V_{n+1}\setminus E} 1_E \ge \eta/2 \on B(y,d_n/2),
\end{equation} 
that is, Brownian motion starting in $B(y,d_n/2)$  hits $E$ with probability at least~$\eta/2$ before leaving $V_{n+1}$.
Since the balls $B(y,d_n/2)$, $y\in Y_n$, cover $\partial V_n$, then Proposition~\ref{nlogn} (this time with $\kappa_n:=\eta/2$) 
will show that $A$ is unavoidable.

To prove the claim let 
$$
B:=B(y,d_n),\qquad D:=B(y, d_n/6), \qquad F:=\ov B(y,d_n/7).
$$
In probabilistic terms we may argue as follows. Starting in $B(y,d_n/2)$, Brownian motion hits $F$ with 
probability at least $\eta$ before leaving $B\subset V_{n+1}$. And, continuing from a point in $F$,
it hits $E$ with probability at least $1/2$ before leaving $D$, by Corollary \ref{y-ball}. So Brownian motion starting
in $B(y,d_n/2)$  hits $E$ with probability at least $\eta/2$ before leaving $V_{n+1}$.

For an analytic proof, we first observe that, by translation and scaling invariance of harmonic measures,
$ H_{B\setminus F} 1_F\ge \eta$ on $B(y,d_n/2)$. By the minimum principle,
$$
H_{V_{n+1}\setminus E}1_E\ge H_{B\setminus E} 1_E\ge  H_{D\setminus E} 1_E,
$$
where $H_{D\setminus E}1_E\ge 1/2$ on $F$, by Corollary \ref{y-ball}, and hence
$$
H_{B\setminus E} 1_E=H_{B\setminus (E\cup F)} H_{B\setminus E} 1_E
\ge  (1/2) H_{B\setminus (E\cup F)} 1_{E\cup F}\ge  (1/2)  H_{B\setminus F} 1_F.
$$
Thus   (\ref{eta-est}) holds and our proof is finished.

\bibliographystyle{plain} %  {alpha}

\def\cprime{$'$} \def\cprime{$'$}

{\small \noindent 
Wolfhard Hansen,
Fakult\"at f\"ur Mathematik,
Universit\"at Bielefeld,
33501 Bielefeld, Germany, e-mail:
 hansen$@$math.uni-bielefeld.de}\\
{\small \noindent Ivan Netuka,
Charles University,
Faculty of Mathematics and Physics,
Mathematical Institute,
 Sokolovsk\'a 83,
 186 75 Praha 8, Czech Republic, email:
netuka@karlin.mff.cuni.cz}

\end{document}